\documentclass[a4paper]{amsart}
\usepackage[utf8]{inputenc}
\usepackage{amssymb,amsmath,enumerate,amsthm,url}
\usepackage[all,cmtip]{xy}
\usepackage{graphicx}
\usepackage[numbers]{natbib}
\usepackage{bibentry}
\usepackage[bookmarks=true]{hyperref}

\usepackage{stackengine}

\usepackage{comment}
\usepackage[colorinlistoftodos]{todonotes}
\usepackage{enumitem}

\def\namedlabel#1#2{\begingroup
    #2%
    \def\@currentlabel{#2}%
    \phantomsection\label{#1}\endgroup
}

\usepackage[nameinlink,capitalize]{cleveref}
\crefformat{equation}{#2(#1)#3}

\theoremstyle{plain}
\newtheorem{theorem}{Theorem}[section]

\newtheorem{lemma}[theorem]{Lemma}
\newtheorem{proposition}[theorem]{Proposition}

\newtheorem{fact}[theorem]{Fact}
\newtheorem*{fact*}{Fact}
\crefname{fact}{Fact}{Facts}
\newtheorem{corollary}[theorem]{Corollary}

\newtheorem*{claim*}{Claim}
\theoremstyle{definition}
\newtheorem{definition}[theorem]{Definition}
\newtheorem*{definition*}{Definition}

\newtheorem*{notation*}{Notation}
\theoremstyle{remark}
\newtheorem{remark}[theorem]{Remark}
\newtheorem*{remark*}{Remark}
\newtheorem{example}[theorem]{Example}
\newtheorem*{example*}{Example}

\newtheorem*{note*}{Note}
\newtheorem{question}[theorem]{Question}
\newtheorem*{question*}{Question}

\newtheorem{problem}[theorem]{Problem}

\numberwithin{equation}{section}

\newcommand{\N}{\mathbb{N}}
\newcommand{\Z}{\mathbb{Z}}
\newcommand{\Q}{\mathbb{Q}}
\newcommand{\Rr}{\mathbb{R}}

\newcommand{\vc}{\operatorname{VC}}
\newcommand{\vcd}{\operatorname{vc}}
\newcommand{\tp}{\operatorname{tp}}

\newcommand{\U}{\mathfrak{C}}

\newcommand{\opp}{\operatorname{opp}}
\newcommand{\Maj}{\operatorname{Maj}}

\providecommand{\Th}{\operatorname{Th}}
\providecommand{\set}[2]{\{#1 \mid #2\}}
\providecommand{\sequence}[2]{(#1)_{#2}}
\providecommand{\varsequence}[2]{(#1\mid #2)}

\renewcommand{\SS}{\mathcal{P}}

\providecommand{\phiopp}{{\phi^{\opp}}}

\newcommand{\defn}[1]{{\emph{#1}}}
\newcommand{\recall}[1]{{\it #1}}

\newcommand{\calF}{\mathcal F}

\providecommand{\eps}{\epsilon}

\newcommand{\card}{\mathbf{Card}}

\newcommand{\ravg}[1]{\operatorname{rAvg}#1}

\newcommand{\Npq}{N_{\operatorname{pq}}}

\begin{document}

\title{A definable $(p,q)$-theorem for NIP theories}
\author{Itay Kaplan}
\address{Einstein Institute of Mathematics, Hebrew University of
Jerusalem, 91904, Jerusalem Israel.}
\email{kaplan@math.huji.ac.il}

\thanks{This research was supported by the Israel Science Foundation (grants no. 1254/18 and 804/22).}


\begin{abstract}
We prove a definable version of Matou\v sek's $(p,q)$-theorem in NIP theories. This answers a question of Chernikov and Simon. We also prove a uniform version.

The proof builds on a proof of Boxall and Kestner who proved this theorem in the distal case, utilizing the notion of locally compressible types which appeared in the work of the author with Bays and Simon. 
\end{abstract}

\maketitle

\section{Introduction}
The aim of this paper is to prove a model-theoretic definable version of Matou\v sek's $(p,q)$-theorem in combinatorics \cite{Matousek} under the assumption that the theory is NIP. 

We begin by recalling Matou\v sek's theorem (a version of which was first proved for convex sets by Alon and Kleitman \cite{Alon-Kleitman}). The statement uses the notion of the (dual) VC-dimension of set systems. For the definitions see \cref{sec:Preliminaries} below. 

\begin{definition}
  Say that a set system $(X,\calF)$ has the \defn{$(p,q)$-property} for $q \leq p <\omega$ if for any $F \subseteq \calF$ of size $|F| = p$ there is some $F_0 \subseteq F$ such that $|F_0|=q$ and $\bigcap F_0 \neq \emptyset$.   
\end{definition}

\begin{fact}[The $(p,q)$-theorem]\cite{Matousek} \label{fact:pq theorem}
  There exists a function $\Npq : \N^2 \rightarrow  \N$ such that for any $q\leq p<\omega$, if $(X,\mathcal{F})$ is a finite set system with the $(p,q)$-property such that every $s\in \mathcal{F}$ is nonempty and $\vc^*(\mathcal{F})<q$, then there is $X_0 \subseteq X$ of size $|X_0| = \Npq(p,q)$ such that $X_0 \cap s \neq \emptyset$ for all $s\in \mathcal{F}$.
\end{fact}

Model theoretically, this implies that if $\phi(x,y)$ is NIP then for any $\vc^*(\phi)<q \leq p$ and $n :=\Npq(p,q)$, if $B$ is a finite set of $y$-tuples such that $\set{\phi(x,b)}{b\in B}$ has the $(p,q)$-property then there are $n$ elements $a_0,\dots,a_{n-1}$ such that for all $b\in B$ there is some $i<n$ for which $\phi(a_i,b)$ holds.

This theorem turned out to be tremendously useful in the model-theoretic study of NIP. For instance, it was instrumental in the proof of the uniform definability of types over finite sets (UDTFS) in NIP theories by Chernikov and Simon \cite{CS-extDef2}, in their study of definably amenable NIP groups \cite{definablyAmenable} and more recently in the proof that honest definitions exist uniformly for NIP formulas \cite{densecompression}. 

In order to phrase a definable version of the $(p,q)$-theorem, we use the following definition.

\begin{definition} \label{def:pair of formulas with the pq property}
  Let $M$ be a structure. Say that a pair of formulas $(\phi(x,y),\psi(y))$ over $M$ has the \defn{$(p,q)$-property} if $\calF := \set{\phi(x,b)}{b\in \psi(M)}$ is a family of nonempty sets with the $(p,q)$-property: for every choice of distinct $p$ elements $\calF$, some $q$ of them have a nonempty intersection. 
\end{definition}

It is not hard to see that for any structure $M$, a pair $(\phi(x,y), \psi(y))$ of formulas over $M$ has the $(p,q)$-property for some $p\geq q>\vc^*(\phi)$ iff for all $b\in \psi(\U)$, $\phi(x,b)$ does not divide over $M$ (this was proved in \cite[Lemma 2.4]{simon-invariant}, but see also \cref{lem:equivalence of non dividing formulas and pq} below). For a reminder of the definition of dividing and forking, see \cref{sec:Preliminaries}. 

In light of this, \cite[Proposition 25]{CS-extDef2} formulates the following corollary of the $(p,q)$-theorem (see also \cite[beginning of Section 2]{simon-dpmin}).

\begin{fact}\label{fact:model theoretically pq}
  Suppose that $T$ is NIP and that $M\vDash T$. Assume that $\phi(x,y)$ and $\psi(y)$ are formulas over $M$ and that $(\phi,\psi)$ has the $(p,q)$-property for $\vc^*(\phi)<q\leq p$. Then there are sets $W_0,\dots,W_{n-1} \subseteq S^y(M)$ for $n:=\Npq(p,q)$ such that $\bigcup_{i<n}W_i = \set{p \in S^y(M)}{\psi(y)\in p}$ and for each $i<n$, $\set{\phi(x,b)}{\tp(b/M)\in W_i}$ is consistent. 
\end{fact}

For example consider the family $\calF$ of rays in DLO (i.e., $\Th(\Q,<)$): sets defined by $x>a$ or $x<a$. It is easy to see that the dual VC-dimension of $\calF$ is $1$ (given any two rays, if they intersect, then their union is everything). In the context of \cref{fact:model theoretically pq}, $\calF$ can be formalized by setting $\phi(x,y,z_1,z_2) = ((z_1 =z_2 \to x>y) \land (z_1 \neq z_2 \to x<y))$ and $\psi(y,z_1,z_2) = (y=y)$. Then $(\phi,\psi)$ has the $(3,2)$ property: every three rays must intersect. Given any model $M$, let $W_0$ be the set of types of pairs $(a,a)$ over $M$ and $W_1$ be the set of types of pairs $(a,b)$ over $M$ where $a\neq b$. In other words, we  cover $\calF$ by positive and negative rays.

In this paper we will prove a definable version of this fact. Here is the statement, essentially taken from \cite[below Fact 2.2]{simon-dpmin}:

\begin{theorem} \label{thm:full definable version intro}
  Suppose that $T$ is NIP, $M\vDash T$ and that $\phi(x,y),\psi(y)$ are formulas over $M$. Assume that $(\phi,\psi)$ has the $(p,q)$-property for $\vc^*(\phi)<q\leq p$. Then there are formulas $\psi_0(y),\dots,\psi_{n-1}(y)$ over $M$ such that $\psi(y)$ is equivalent to the disjunction $\bigvee_{i<n} \psi_i(y)$ and for each $i<n$, $\set{\phi(x,b)}{b \in \psi_i(M)}$ is consistent. 
\end{theorem}

This is \cref{cor:definable p q explanation}. In the example above this is illustrated by taking $\psi_0 = (z_1=z_2)$ and $\psi_1 = (z_1 \neq z_2)$. 

The above theorem follows by compactness from the following theorem. 

\begin{theorem} \label{thm:the main theorem intro}
  Suppose that $T$ is NIP and that $M \vDash T$. Suppose that $\phi(x,y)$ is a formula over $M$ and $b\in \U^y$ is such that $\phi(x,b)$ does not fork over $M$. Then there is a formula $\psi(y)\in \tp(b/M)$ such that $\set{\phi(x,c)}{\U\vDash \psi(c)}$ is consistent. 
\end{theorem}

This is \cref{thm:the main theorem}. The idea of the proof is to generalize the argument of \cite{DistalDefinablePQ} by Boxall and Kestner who proved this theorem in the case when $T$ is distal (a subclass of NIP theories introduced by Simon in \cite{pierre-distal}). Their proof uses the fact that $T$ is distal only once, in \cite[Proposition 4.1]{DistalDefinablePQ}. We generalize this proposition to the NIP case by using the notion of locally compressible types from \cite{densecompression}, and being more careful with the choice of the strict Morley sequence.  

We also consider a uniform version of \cref{thm:full definable version intro}, i.e., varying the model $M$. 

\begin{theorem} \label{thm:uniform definable pq intro}
  Suppose that $T$ is NIP, and that $\phi'(x,y,z)$, $\psi'(y,z)$ are two formulas without parameters.  Then for any $q \leq p <\omega$ there is $n<\omega$ and formulas $\psi_0(y,w),\dots,\psi_{n-1}(y,w)$ such that the following hold.
  
  Suppose that $M\vDash T$ and $c \in M^z$. Let $\phi(x,y) = \phi'(x,y,c)$ and $\psi(y) = \psi'(y,c)$. If $(\phi,\psi)$ has the $(p,q)$-property and $\vc^*(\phi(x,y)) < q$ then for some $d_0,\dots,d_{n-1} \in M^w$, $\psi(y,c)$ is equivalent to the disjunction $\bigvee_{i<n}\psi_i(y,d_i)$ and for each $i<n$, the set $\set{\phi(x,b)}{b\in \psi_i(M,d_i)}$ is consistent. 
\end{theorem}

This is \cref{thm:uniform definable pq}. We deduce a uniform version related to the more refined notion of VC-density (used in Matou\v sek's original formulation of \cref{fact:pq theorem}, see \cref{fact:Matousek original}) in \cref{cor:Pablo's uniformity}, answering positively a version of \cite[Questions 3.7(2)]{Guerrero} where $T$ is assumed to be NIP, without restrictions on the VC-codensity. 


\subsection*{A short history of the problem}
\cref{thm:the main theorem intro} was first posed as a question by Chernikov and Simon in \cite[Problem 29]{CS-extDef2} even before the relation to the $(p,q)$-theorem was noticed. It was later conjectured by Simon \cite[Conjecture 5.1]{simon-invariant}. 

This was settled in the following cases:
\begin{itemize}
  \item In \cite{simon-dpmin}, Simon proved it for dp-minimal theories with small or medium directionality (a notion which measures the number of coheirs, see \cite{Sh946}).
  \item In \cite{SimonStarchenko}, Simon and Starchenko prove a stronger version of this conjecture for a large class of dp-minimal theories (e.g., those with definable Skolem functions), namely that every such formula $\phi(x,b)$ belongs to a definable type.
  \item In \cite{simon-invariant}, Simon generalized the first item to NIP theories (still constraining the directionality). In both papers \cite{simon-dpmin,simon-invariant} there are very interesting discussions of this problem and related results. 
  \item In \cite{DistalDefinablePQ}, Boxall and Kestner proved the conjecture for distal theories.
  \item In \cite{Rakotonarivo}, Rakotonarivo proved the conjecture for certain dense pairs of geometric distal structures.
  \item In \cite{Guerrero}, And\'ujar-Guerrero proved a special case of a stronger conjecture, \cite[Conjecture 2.15]{simon-invariant} which assumes only that $\phi$ is NIP, namely the case in which the VC-codensity of $\phi$ is less than $2$.
\end{itemize}

\subsection*{Structure of the paper}
In \cref{sec:Preliminaries} we go over all the basic notions involved in the proof, including NIP and forking. In \cref{sec:the proof} we prove \cref{thm:the main theorem intro}. In \cref{sec:uniformity} we prove \cref{thm:uniform definable pq intro}, and in \cref{subsec:VC-codensity uniformity} we review the notion of VC-density and deduce a variant of \cref{thm:uniform definable pq} corresponding to this notion. We conclude in \cref{sec:final thoughts} with some questions and final thoughts. 

\subsection*{Acknowledgements}
I would like to thank Martin Bays and Pierre Simon for many useful discussions and for their comments on previous versions of this paper. I would like to thank Pierre Simon in particular for reminding me of the direct proof of \cref{lem:compressible eventually positive type}. 

I would also like to thank Pablo And\'ujar-Guerrero for his comments on a previous version, for encouraging me to relate the results to VC-density, and for pointing out \cref{lem:pq for some q>vc*}.

\section{Preliminaries} \label{sec:Preliminaries}
Our notation is standard and is the same as in \cite[Section 2]{densecompression}. The only exception is that we use the notation $\U$ for the monster model. 

In the following subsections we will recall the basic definitions and facts we will use in the rest of the paper.

\subsection{VC-dimension and NIP}\label{subsec:VC-dim and NIP}

\begin{definition} [VC-dimension]
  Let $X$ be a set and $\calF\subseteq\SS(X)$. The pair $(X,\calF)$ is called a \defn{set system}.
  We say that $A\subseteq X$ is \defn{shattered} by $\calF$ if for every $S\subseteq A$ there is $F\in\calF$ such that $F\cap A=S$. A family $\calF$ is said to be a \defn{VC-class} on $X$ if there is some $n<\omega$ such that no subset of $X$ of size $n$ is shattered by $\calF$. In this case the \defn{VC-dimension of $\calF$}, denoted by $\vc(\calF)$, is the smallest integer $n$ such that no subset of $X$ of size $n+1$ is shattered by $\calF$.

  If no such $n$ exists, we write $\vc(\calF)=\infty$.
\end{definition}

\begin{fact} \cite[Lemma 6.3]{simon-NIP} \label{fact:dual}
  Suppose $\mathcal{F}$ is a VC-class on $X$. Let $\mathcal{F}^* = \set{\set{s \in \mathcal F}{x \in s}}{x \in X} \subseteq \SS(\mathcal{F})$ be the \defn{dual} of $\mathcal{F}$. Then $\mathcal{F}$ is a VC-class iff $\mathcal{F}^*$ is, and moreover $\vc^*(\mathcal{F}):=\vc(\mathcal{F}^*)<2^{\vc(\mathcal{F})+1}$.
\end{fact}

\begin{definition} \label{def:NIP formula and theory}
  Suppose $T$ is an $\mathcal{L}$-theory and $\phi(x,y)$ is a formula.
  Say that $\phi(x,y)$ is \defn{NIP} if for some/every $M\vDash T$, the family $\set{\phi(M,a)}{a\in M^{y}}$ is a VC-class. Otherwise, $\phi$ is \defn{IP} (IP stands for ``Independence Property'' while NIP stands for ``Not IP'').

  Let $\vc(\phi)$ be the VC-dimension of $\set{\phi(M,a)}{a\in M^{y}}$, where $M$ is any (some) model of $T$.  Note that this definition depends on the partition of variables.

  Let $\phiopp$ be the partitioned formula $\phi(y,x)$ (it is the same formula with the partition reversed). Let $\vc^*(\phi) = \vc(\phiopp)$ be the \defn{dual} VC-dimension of $\phi$ (this definition agrees with the one in \cref{fact:dual} below). 

  The (complete first-order) theory $T$ is \defn{NIP} if all formulas are NIP. A structure $M$ is NIP if $\Th(M)$ is NIP.
\end{definition}

\begin{remark} \label{rem:alternation rank}
  NIP formulas are especially well-behaved in the presence of indiscernible sequences. Indeed, for $\phi(x,y)$ NIP and an indiscernible sequence $\sequence{a_i}{i<\omega}$ of $y$-tuples there is no $b \in \U^x$ such that $\U \vDash \neg (\phi(b,a_i) \leftrightarrow \phi(b,a_{i+1}))$ for all $i<2\vc^*(\phi)+1$ (see \cite[Proposition 3]{Adler} or \cite[Lemma 2.7]{simon-NIP}). The maximal such $l$ is called the \defn{alternation rank of $\phi$}.  
\end{remark}

From this one deduces the following facts.

\begin{fact} \label{fac:alternation} 
  Let $\phi(x,y)$ be an NIP formula. Suppose that $I = \sequence{a_i}{i<\omega}$ is an indiscernible sequence of $y$-tuples.
  \begin{enumerate}
    \item \label{enu:eventual} \cite[Proposition 2.8]{simon-NIP} For any $b \in \U^x$ there is some $n<\omega$ and $\epsilon<2$ such that $\phi(b,a_i)^\epsilon$ holds for all $i>n$. 
    \item \label{enu:low} (``Lowness'') \cite[Lemma 2.2]{simon-invariant} There is some $n<\omega$ such that if $\set{\phi(x,a_i)}{i<n}$ is consistent, then so is $\set{\phi(x,a_i)}{i<\omega}$. In fact, $n$ can be chosen to be $\vc^*(\phi)+1$. 
  \end{enumerate}
\end{fact}

\subsection{Locally compressible types}

\newcommand{\down}{\downarrow}
\begin{definition} \label{def:local compression}
  Fix a formula $\phi(x,y)$, $k<\omega$ and a parameter set $A \subseteq  \U^y$.
  \begin{itemize}
    \item $p \in S_\phi$ is \defn{$k$-compressible} if for any finite $A_0 \subseteq  A$ there is $A_1 \subseteq  A$ with $|A_1| \leq  k$ such that $p|_{A_1}(x) \vdash p|_{A_0}(x)$.

    \item $S_{\phi \down k}(A) \subseteq  S_\phi(A)$ is the space of $k$-compressible $\phi$-types.
  \end{itemize}
\end{definition}

\begin{definition}\label{def:rounded average}
  Suppose $\phi(x,y)$ is a formula, $B \subseteq \U^y$ and $p_0(x),\ldots ,p_{n-1}(x) \in S_\phi(B)$.

  The \defn{rounded average} of  $p_0(x),\ldots ,p_{n-1}(x) \in S_\phi(B)$ is the following (possibly inconsistent) collection of formulas 
  \[\ravg{\varsequence{p_i}{i<n}} =\set{\phi(x,b)^\eps}{b \in B, \eps<2, \Maj_{i<n} (\phi(x,b)^\eps \in p_i(x))},\]

  where $\Maj$ be the majority rule Boolean operator: for truth values $P_0,\dots,P_{n-1}$, let
  \[\Maj_{i < n} P_i =
  \bigvee_{\substack{I_0 \subseteq n \\ |I_0| > n/2}}  \bigwedge_{i \in I_0} P_i.\]
\end{definition}

\begin{fact} \label{fact:rounded average of compressible} \cite[Theorem 5.17]{densecompression}
  Let $T$ be any theory. Let $\phi(x,y)$ be an NIP formula.
  Then there exist $n$ and $k$ depending only on $\vc(\phi)$ such that for $A \subseteq  \U^y$,
  any $p \in S_{\phi}(A)$ is the rounded average of $n$ types in 
  $S_{\phi\down k}(A)$.
\end{fact}

\subsection{Forking and NIP}
We recall the definition of forking in any theory $T$ and its behavior under NIP. 
\begin{definition} [Forking]
  Fix a set $A$.

  A formula $\phi(x,b)$ \defn{divides} over $A$ if there is some $k<\omega$ such that $\phi(x,b)$ $k$-divides: for some sequence $\sequence{b_i}{i<\omega}$ such that $b_i \equiv_A b$, $\set{\phi(x,b_i)}{i<\omega}$ is $k$-inconsistent. 

  A formula $\psi(x)$ \defn{forks} over $A$ if it implies a finite disjunction of dividing formulas over $A$.

  A partial type $\pi(x)$ \defn{divides/forks} over $A$ if it implies a dividing/forking formula over $A$.
\end{definition}

For more on the basic properties of forking and dividing, see e.g., \cite[Section 2]{forkingNTP2}. 

\subsubsection{Forking and the $(p,q)$-theorem} \label{subsec:forking and pq}

We now turn to the $(p,q)$-theorem (\cref{fact:pq theorem}) and its connection to non-forking. Recall \cref{def:pair of formulas with the pq property} from the introduction.

Applying \cref{fact:pq theorem}, we get directly the following.
\begin{lemma} \label{lem:pq for pairs of formulas}
  Suppose that $M \vDash T$ and that $\phi(x,y)$, $\psi(y)$ are formulas over $M$.  Suppose that $(\phi, \psi)$ has the $(p,q)$-property for $\vc^*(\phi) < q \leq p$ (in particular, $\phi$ is NIP). Let $n = \Npq(p,q)$. Then for every finite set $C \subseteq \psi(M)$, there are $a_0,\dots,a_{n-1} \in M^x$ such that for any $c\in C$, $\phi(a_i,c)$ holds for some $i<n$. 
\end{lemma}

The following was proved in \cite[Lemma 2.4]{simon-invariant} over models, but the proof there omitted the (necessary) condition that the $\phi(M,b)\neq \emptyset$ for all $b\in \psi(M)$, so we repeat it briefly here. 

\begin{lemma} \label{lem:equivalence of non dividing formulas and pq}
  Suppose that $M \vDash T$ and that $\phi(x,y)$, $\psi(y)$ are formulas over some set $A \subseteq M$ with $\phi$ NIP. Then the following are equivalent:
  \begin{enumerate}
    \item For all $b\in \psi(\U)$, $\phi(x,b)$ does not divide over $A$.     
    \item For all $0<q$ there is some $p \geq q$ such that $(\phi, \psi)$ has the $(p,q)$-property.
    \item $(\phi, \psi)$ has the $(p,q)$-property for some $\vc^*(\phi)<q \leq p$. 
  \end{enumerate} 
\end{lemma}

\begin{proof}
  Note that the statement ``$(\phi, \psi)$ has the $(p,q)$-property'' is elementary, so it is true in $M$ iff it is true in $\U$. 
  
  For (1) implies (2), note that (1) implies immediately that for all $b\in \psi(\U)$, $\phi(\U,b)$ is nonempty. Fix $q>0$. If for every $p\geq q$, $(\phi, \psi)$ does not have the $(p,q)$-property then by compactness and Ramsey there is an $A$-indiscernible sequence $\sequence{b_i}{i<\omega}$ of elements from $\psi(\U)$ witnessing that $\phi(x,b_0)$ divides over $A$. 

  (2) implies (3) is trivial. 

  Assume (3). Suppose that for some $b \in \psi(\U)$, $\phi(x,b)$ divides over $A$. Then there is an $A$-indiscernible sequence $\sequence{b_i}{i<\omega}$ starting with $b_0 =b$ such that $\set{\phi(x,b_i)}{i<\omega}$ is inconsistent. By \cref{rem:alternation rank}, it is already $\vc^*(\phi)+1 \leq q$-inconsistent. Note that for each $i<j<\omega$, $\phi(\U,b_i) \neq \phi(\U,b_j)$ since otherwise by indiscernibility, $\sequence{\phi(\U,b_i)}{i<\omega}$ is constant and thus all these sets will be empty. Together, we get that $\sequence{\phi(\U,b_i)}{i<p}$ contradicts the $(p,q)$-property. 
\end{proof}

Finally, recall the following corollary of ``lowness'' of NIP formulas (\cref{fac:alternation}(\ref{enu:low})):

\begin{fact}\label{fac:dividing type definable} \cite[Corollary 2.3]{simon-invariant}
  Suppose that $\phi(x,y)$ is a NIP formula (without parameters), and that $\phi(x,b)$ does not divide over a model $M$. Then there is a formula $\psi(y) \in \tp(b/M)$ such that for all $b' \in \psi(\U)$, $\phi(x,b')$ does not divide over $M$. 

  Together with \cref{lem:equivalence of non dividing formulas and pq}, we get that $(\phi,\psi)$ has the $(p,q)$-property for $q:=\vc^*(\phi)+1$ and some $p\geq q$. 
\end{fact} 

\subsubsection{Forking in NIP and NTP$_2$} \label{subsec:forking in NIP and NTP2}

The basic property of non-forking in NIP is the equivalence with Lascar-invariance, which over models translates to invariance. 
\begin{fact} \label{fac:forking in NIP} (\cite[Proof of Proposition 5.21]{simon-NIP})
  If $\phi(x,y)$ is NIP and $p(x) \in S_\phi(\U)$ is a global $\phi$-type, non-forking over a model $M$, then $p$ is $M$-invariant. 
  
  It follows that if $\phi(x,c)$ does not fork over $M$, and $B \subseteq \U^y$ is a set of realizations of $\tp(c/M)$ then $\set{\phi(x,b)}{b\in B}$ does not fork over $M$ (and is in particular consistent). 

\end{fact}

When $T$ is NIP or even NTP$_2$ --- a larger class which contains also simple theories (for the definition, see \cite[Definition 2.27]{forkingNTP2}), we have that:

\begin{fact}\label{fact:forking=dividing} \cite{forkingNTP2}
  Assume that $T$ is NTP$_2$ and let $M \vDash T$. Then forking equals dividing over $M$: if $\phi(x,y)$ is a formula over $M$ and $\phi(x,b)$ forks over $M$, then it divides over $M$. 
\end{fact}

One of the main tools used to show this was the existence of global strictly invariant types since Morley sequences in them are universal witnesses for dividing.

\begin{definition}
  A global type $w(y) \in S(\U)$ is \defn{strictly invariant} over a model $M$ if $w(y)$ is invariant over $M$ and for any set $C \supseteq M$, if $a \vDash w|_C$ then $\tp(C/Ma)$ does not fork over $M$.

  When $(I,<)$ is a linear order and $\sequence{b_i}{i \in I}$ is a Morley sequence of an $M$-strictly invariant type $w$ over $M$ (see e.g., \cite[Section 2.2]{densecompression} for the definition of a Morley sequence), we call it a \defn{strict Morley sequence} over $M$.

\end{definition}

\begin{fact}\cite[Lemma 3.14]{forkingNTP2}\label{fact:strict MS witnesses}
  Assume that $T$ is NTP$_2$, $M \vDash T$ and $\phi(x,y)$ a formula over $M$. Suppose that for some $b \in \U^y$, $\phi(x,b)$ divides over $M$.
  
  Then, if $\sequence{b_i}{i<\omega}$ is a strict Morley sequence over $M$, then $\set{\phi(x,b_i)}{i<\omega}$ is inconsistent.  
\end{fact}

We will need a strict invariant type which is also a coheir. The existence of such a type is implied by \cite[Proposition 3.7 (1)]{forkingNTP2}, applied to coheir independence, using the fact that forking over models implies quasi-dividing in NTP$_2$, see \cite[Corollary 3.13]{forkingNTP2} which is precursor to \cref{fact:forking=dividing}. We state this explicitly:

\begin{fact} \label{fact:non-coforking coheir}
  Assume that $T$ is NTP$_2$ and let $M \vDash T$. Then for any type $q(y) \in S(M)$ there is a global type $w(y) \in S(\U)$ extending $q(y)$ such that:
  \begin{itemize}
    \item $w$ is finitely satisfiable in $M$. 
    \item If $C \supseteq M$ and $a \vDash w|_C$ then $\tp(C/Ma)$ does not fork over $M$. 
  \end{itemize}
\end{fact}

\section{The proof of \texorpdfstring{\cref{thm:the main theorem intro}}{Theorem 1.6}} \label{sec:the proof}
We start with the following lemma. 
\begin{lemma} \label{lem:compressible eventually positive type}
  Suppose that $\phi(x,y)$ is NIP. Then there is some $k$ depending only on $\vc(\phi)$ such that if $I:=\sequence{a_i}{i<\omega}$ is an $\emptyset$-indiscernible sequence of $y$-tuples such that $p:=\set{\phi(x,a_i)}{i<\omega}$ is consistent, then there is a $k$-compressible type $q \in S_{\phi \down k}(I)$ such that for some $n<\omega$, for all $i>n$, $\phi(x,a_i)\in q$. 
\end{lemma}

\begin{proof}
  By \cref{fact:rounded average of compressible}, $p$ is the rounded average of $n$ types in $S_{\phi\down k}(A)$ for some $k,n$ which depend only on $\vc(\phi)$. 
  By \cref{fac:alternation}(\ref{enu:eventual}), each of these types must be eventually positive or eventually negative. However, since $p$ is eventually positive (in fact all the instances of $\phi$ in $p$ are positive), and $p$ is the rounded average of these types, at least one (and in fact the majority) of them has to be eventually positive. 
  
  Here is a more direct argument which does not use \cref{fact:rounded average of compressible} (this argument was used in its proof\footnote{We thank Pierre Simon for reminding us of the existence of a more direct proof}). The idea is to take a $\phi$-type which first alternates maximally and then is constantly true. More precisely, let $m= 2\vc^*(\phi)+3$. Let $l<m$ be maximal such that there is some $q_0(x) \in S_\phi(\set{a_i}{i<m})$ such that $q_0 \vdash \phi(x,a_i) \leftrightarrow  \neg \phi(x,a_{i+1})$ for $i < l$
  and $q_0(x) \vdash  \phi(x,a_i)$ for $i \in [l,m)$,
  
  Note that $p|_{a_{<m}}$ is an example of such a type with $l=0$ so that such an $l$ and $q_0$ exist. By \cref{rem:alternation rank}, $l \leq  2\vc^*(\phi) < m-2$. 

  Let $q = q_0 \cup \set{\phi(x,a_i)}{l\leq i}$. Then $q$ consistent and moreover $(m-1)$-compressible. 
  
  Indeed, let $l< n<\omega$. We show that $q|_{a_{\leq l}\cup \{a_{n},\dots,a_{n+m-l-3}\}} \vdash q|_{a_{<n}}$. Otherwise, by indiscernibility, $q|_{a_{\leq l}\cup \{a_{l+2},\dots,a_{m-1}}\} \cup \{\neg \phi(x,a_{l+1})\}$ is consistent, contradicting the choice of $l$. Finally, note that the left-hand side of this implication is consistent by indiscernibility (and since $q_0$ is consistent).  
\end{proof}

The following is a generalization of \cite[Proposition 4.1]{DistalDefinablePQ}.
\begin{proposition} \label{prop:lemma as in Gareth Lotte}
  Suppose that $T$ is NIP, $M\vDash T$, $\phi(x,y)$ a formula over $M$, $b\in \U^y$ and that $\phi(x,b)$ does not fork over $M$. Let $q(y) = \tp(b/M)$.
  
  Then there is a formula $\theta(x,z)$ over $M$ and a type $r(z)\in S(M)$ such that:
  \begin{enumerate}
    \item For any $c \vDash r$, $\theta(x,c)$ does not fork over $M$.
    \item For any finite set $B$ of realizations of $q$, there is some $c \vDash r$ such that $\U\vDash \forall x (\theta(x,c) \to \phi(x,b'))$ for any $b' \in B$. 
  \end{enumerate}

  \begin{proof}
    By \cref{fact:non-coforking coheir} there is a global type $w(y) \in S(\U)$ extending $q(y)$ such that:
    \begin{itemize}
      \item $w$ is finitely satisfiable in $M$. 
      \item If $C \supseteq M$ and $a \vDash w|_C$ then $\tp(C/Ma)$ does not fork over $M$. 
    \end{itemize}
    Let $J = \sequence{a_i}{i \in \Z}$ be a Morley sequence of $w$ over $M$, i.e., $a_i \vDash w |_{Ma_{<i}}$ for all $i \in \Z$. For $i < \omega$, let $b_i = a_{-i}$ and let $I = \sequence{b_i}{i<\omega}$. Note that for all $i<\omega$:
    \begin{enumerate}[label={(\roman*)}]
      \item \label{enu:first part fs}$\tp(b_{<i}/Mb_{\geq i})$ is finitely satisfiable in $M$ (by transitivity of coheir independence). 
      \item \label{enu:second part nf} $\tp(b_{\geq i}/Mb_i)$ does not fork over $M$. 
    \end{enumerate}
    Let $p = \set{\phi(x,b_i)}{i<\omega} \in S_\phi(I)$. As $I$ is an $M$-indiscernible sequence, $b_0 \equiv_M b$ and as $\phi(x,b)$ does not fork over $M$, $p$ is consistent. 
    
    By \cref{lem:compressible eventually positive type} applied to $p$, there is some $k<\omega$ and some $r \in S_{\phi\down k}(I)$ such that $r$ is eventually positive: for some $s<\omega$, $\phi(x,b_i) \in r$ for all $i>s$. (In fact, the direct proof of \cref{lem:compressible eventually positive type} allows us to set $k=s=2\vc^*(\phi)+2$.)
    
    
    Thus, for some $A \subseteq \omega$ of size $|A|=k$, $r|_{\set{b_i}{i\in A}} \vdash r|_{b_{\leq s+k+1}}$. In particular, for some $s< j \leq s+k+1$, $j \notin A$.
    
    Write $A = A_0 \cup A_1$ where $A_0 = A \cap j$ and $A_1 = A \setminus A_0$. By indiscernibility and the choice of $s$, we get that $r|_{\set{b_i}{i\in A_0} \cup \set{b_i}{j<i\leq j+t}} \vdash \phi(x,b_j)$ where $t = |A_1|$. 
  
    Let $u = \sequence{u_i}{i\in A_0}$ be a tuple of variables of the sort $y$, and let $\theta_0(x,u) = \bigwedge_{i \in A_0} \phi(x,u_i)^{(\phi(x,b_i) \in r)}$. Let $z = \sequence{z_i}{i<t}$ be a sequence of variables of the sort $y$ and let $\theta_1(x,z) = \bigwedge_{i<t} \phi(x,z_i)$. We claim that letting $r(z) = \tp(b_{<t}/M)$, there is some $m \in M^{u}$ such that the pair $r(z)$ and $\theta_m(x,z) := \theta_0(x,m)\land \theta_1(x,z)$ satisfies the conclusion of the proposition. 

    By ``lowness'' (\cref{fac:alternation}(\ref{enu:low})), there is some number $f <\omega$ such that for any $m \in M^u$, if $\set{\theta_0(x,m)\land \phi(x,b_i)}{i<f}$ is consistent, then $\set{\theta_0(x,m)\land \phi(x,b_i)}{i<\omega}$ is consistent (in fact $f$ could be chosen to be $\vc^*(\phi)+1$). 
    Fix such an $m \in M^u$. Since strict Morley sequences are universal witnesses for dividing (\cref{fact:strict MS witnesses}) and $J$ is a strict Morley sequence over $M$, $\theta_0(x,m) \land \phi(x,b_0)$ does not divide over $M$. Since forking equals dividing over $M$ by \cref{fact:forking=dividing} and non-forking over models is the same as invariance by \cref{fac:forking in NIP}, it follows that $\theta_m(x,b_{<t})$ does not fork over $M$. 

    Recall that by \cref{enu:first part fs}, $\tp(\sequence{b_i}{i\in A_0}/Mb_{>j})$ is finitely satisfiable in $M$. Thus, there is $m\in M^u$ such that:
    \begin{itemize}
      \item $\theta_m(x,\sequence{b_i}{j+1\leq i < j+1+t}) \vdash \phi(x,b_j)$.
      \item $\theta_0(x,m)\land \bigwedge_{j+1\leq i < j+1+f} \phi(x,b_i)$ is consistent. 
    \end{itemize}
    
    Thus, by the previous paragraph and indiscernibility, letting $\theta(x,z) = \theta_m(x,z)$, $\theta(x,b_{<t})$ does not fork over $M$.
    
    By indiscernibility, it follows that $\theta(x,b_{1\leq i < t+1}) \vdash \phi(x,b_0)$. Recall that by \cref{enu:second part nf}, $\tp(b_{1\leq i < t+1}/Mb_0)$ does not fork over $M$. Now, given a finite set $B$ of realization of $q(y)$, by extension, there is some $c \vDash r(z)$ such that $\tp(c/MBb_0)$ does not fork over $M$. As non-forking over models is the same as invariance by \cref{fac:forking in NIP}, we get that $\theta_m(x,c) \vdash \phi(x,b')$ for all $b'\in B$ as required. 
  \end{proof}
\end{proposition}

We now deduce \cref{thm:the main theorem intro} which we repeat below. The proof is exactly the same as in \cite[after Proposition 4.1]{DistalDefinablePQ}, but we break it down for use in \cref{sec:uniformity}.

\begin{proposition}\label{prop:pq + implication => consistent}
  Let $M$ be any structure and let $\phi(x,y),\psi(y),\theta(x,z),\zeta(z)$ be formulas over $M$. Suppose that:
  \begin{enumerate}
    \item $\theta(x,z)$ is NIP. 
    \item $(\theta,\zeta)$ has the $(p,q)$-property for some $\vc^*(\theta(x,z))< q \leq p$ and let $n = \Npq(p,q)$.
    \item \label{enu:implication} For every set $B \subseteq \psi(M)$ such that $|B|\leq n$, there is some $c\in \zeta(M)$ such that $M\vDash \forall x (\theta(x,c)\to \phi(x,b))$ for all $b\in B$.
  \end{enumerate}
  Then $\set{\phi(x,b)}{b\in \psi(M)}$ is consistent. 
\end{proposition}

\begin{proof}
  By \cref{lem:pq for pairs of formulas} we get that for any finite subset $C \subseteq \zeta(M)$ there are $a_0,\dots,a_{n-1} \in M^x$ such that for all $c \in C$, $\theta(a_i,c)$ holds for some $i<n$. By compactness, there are $a_0,\dots,a_{n-1} \in \U^x$ such that for all $c\in \zeta(M)$, $\theta(a_i,c)$ holds for some $i<n$. 
  
  We claim that for some $i<n$, $\phi(a_i,b')$ holds for all $b'\in \psi(M)$. Otherwise, for each $i<n$ there is $b_i \in \psi(M)$ such that $\neg \phi(a_i,b_i)$. Let $B = \set{b_i}{i<n}$, and let $c$ be as in (\ref{enu:implication}). By the previous paragraph, for some $i<n$, $\theta(a_i,c)$ holds, and thus $\phi(a_i,b')$ for all $b'\in B$ and in particular $\U \vDash \phi(a_i,b_i)$, a contradiction.
\end{proof}

\begin{theorem} [\cref{thm:the main theorem intro}]\label{thm:the main theorem}
  Suppose that $T$ is NIP and that $M \vDash T$. Suppose that $\phi(x,y)$ is a formula over $M$ and $b\in \U^y$ is such that $\phi(x,b)$ does not fork over $M$. Then there is a formula $\psi(y)\in \tp(b/M)$ such that $\set{\phi(x,c)}{\U\vDash \psi(c)}$ is consistent. 
\end{theorem}

\begin{proof}
  Let $q(y) = \tp(b/M)$. By \cref{prop:lemma as in Gareth Lotte}, we get some $\theta(x,z)$ and $r(z) \in S(M)$ as in there. As $\theta(x,c)$ does not fork over $M$ for all $c\vDash r$, by \cref{fac:dividing type definable} there is a formula $\zeta(z)$ in $r(z)$ such that for all $c'\vDash \zeta(z)$, $\theta(x,c')$ does not divide over $M$. By \cref{lem:equivalence of non dividing formulas and pq}, there is some $p>q:=\vc^*(\theta(x,z))+1$ such that $(\theta,\zeta)$ has the $(p,q)$-property. 
  
  Let $n:=\Npq(p,q)$. By compactness and the choice of $\theta$ (and the fact that $\zeta \in r$), there is a formula $\psi(y) \in q$ such that for any set $B \subseteq \psi(M)$ of size $|B|\leq n$, there is some $c \in \zeta(M)$ such that $M \vDash \forall x (\theta(x,c) \to \phi(x,b'))$ for any $b' \in B$.

  Applying \cref{prop:pq + implication => consistent}, we are done.
 \end{proof}

\begin{corollary}
  Under the same assumptions as in \cref{thm:the main theorem}, there is a global type $p(x)$ and a formula $\psi(y)\in \tp(b/M)$ such that $p(x)\supseteq \set{\phi(x,c)}{\U\vDash \psi(c)}$ and $p$ is a non-forking heir over $M$.
\end{corollary}
\begin{proof}
  By \cref{thm:the main theorem}, there is a formula $\psi(y) \in \tp(b/M)$ such that $\Gamma(x):=\set{\phi(x,b)}{b\in \psi(M)}$ is consistent. Let $p_0(x)\in S(M)$ extend $\Gamma(x)$ to a complete type over $M$. By \cite[Proposition 3.7 (2)]{forkingNTP2} applied to coheir independence, there is a non-forking heir $p(x)\in S(\U)$ extending $p_0$. It follows from the heir property that $p(x)$ contains $\set{\phi(x,c)}{\U\vDash \psi(c)}$. 
\end{proof}

The following was already essentially stated in \cite[beginning of Section 2]{simon-dpmin}:

\begin{corollary}[\cref{thm:full definable version intro}] \label{cor:definable p q explanation}
  Suppose that $T$ is NIP, $M\vDash T$ and that $\phi(x,y),\psi(y)$ are formulas over $M$. Assume that $(\phi,\psi)$ has the $(p,q)$-property for $\vc^*(\phi)<q\leq p$. Then there are formulas $\psi_0(y),\dots,\psi_{n-1}(y)$ over $M$ such that $\psi(y)$ is equivalent to the disjunction $\bigvee_{i<n} \psi_i(y)$ and for each $i<n$, $\set{\phi(x,b)}{b \in \psi_i(M)}$ is consistent. 
\end{corollary}

\begin{proof}
  By \cref{lem:equivalence of non dividing formulas and pq}, for every $b \in \psi(\U)$, $\phi(x,b)$ does not fork over $M$. \cref{thm:the main theorem} implies that for every $b\in \psi(\U)$ there is a formula $\psi_b(y) \in \tp(b/M)$ such that $\set{\phi(x,b)}{b\in \psi_b(M)}$ is consistent. By compactness, finitely many such formulas cover all of $\psi(y)$ as required. 
\end{proof}

\section{Uniformity} \label{sec:uniformity}
Now we want to give a uniform version of \cref{cor:definable p q explanation}.

\begin{theorem} [\cref{thm:uniform definable pq intro}] \label{thm:uniform definable pq}
  Suppose that $T$ is NIP, and that $\phi'(x,y,z)$, $\psi'(y,z)$ are two formulas without parameters.  Then for any $q \leq p <\omega$ there is $n<\omega$ and formulas $\psi_0(y,w),\dots,\psi_{n-1}(y,w)$ such that the following hold.
  
  Suppose that $M\vDash T$ and $c \in M^z$. Let $\phi(x,y) = \phi'(x,y,c)$ and $\psi(y) = \psi'(y,c)$. If $(\phi,\psi)$ has the $(p,q)$-property and $\vc^*(\phi(x,y)) < q$ then for some $d_0,\dots,d_{n-1} \in M^w$, $\psi(y)$ is equivalent to the disjunction $\bigvee_{i<n}\psi_i(y,d_i)$ and for each $i<n$, the set $\set{\phi(x,b)}{b\in \psi_i(M,d_i)}$ is consistent. 
\end{theorem}

\begin{remark} \label{rem:example where naive nf uniform version doesnt work}
  Based on \cref{lem:equivalence of non dividing formulas and pq}, one could ask if the following holds: 

  Suppose that $T$ is NIP, and that $\phi(x,y,z)$, $\psi(y,z)$ are two formulas (without parameters). Then there is $n<\omega$ and formulas $\psi_0(y,w_0),\dots,\psi_{n-1}(y,w_{n-1})$ such that for any $M\vDash T$ and any $c \in M^z$, if $\phi(x,b,c)$ does not fork over $M$ for any $b \in \psi(M,c)$ then for some $d_0,\dots,d_{n-1} \in M$, $\psi(y,c)$ is equivalent to the disjunction $\bigvee_{i<n}\psi_i(y,d_i)$ and for each $i<n$, the set $\set{\phi(x,b)}{b\in \psi_i(M,d_i)}$ is consistent. 

  However, consider the following example. Let $\mathcal{L} = \{E\}$ and let $T$ say that $E$ is an equivalence relation such that for every $n<\omega$ there is exactly one class with $n$ elements. Then $T$ is a complete stable theory. Let $\psi(y,z) = (y \mathrel{E} z)$ and $\phi(x,y,z) = (x=y)$. Suppose that $\psi_i$ for $i<n$ are as above. Fix some model $M$. Let $c$ be such that $n<|[c]_E|<\omega$. Then $(x=b)$ does not fork over $M$ for any $b \mathrel{E} c$. By the choice of $\psi_i$ for $i<n$, there are $d_0,\dots,d_{n-1}$ as above. But then there must be some $i<n$ such that $2\leq |\psi_i(M,d_i)|$ and $\set{x=b}{b\in \psi_i(M,d_i)}$ is inconsistent. 
\end{remark}

\begin{remark}\label{rem:pseudofinite case}
  Before giving the proof, consider a special case, where $T$ is pseudofinite. In that case, the formulas $\psi_i$ can be taken to be of the form $\phiopp(y,w) \land \psi(y)$ and $n = \Npq(p,q)$. The reason is that given any $M,c$ as above, there are $d_i \in M^x$ for $i<n$ such that for all $b \in \psi(M,c)$, $\phi(d_i,b)$ holds for some $i<n$. Indeed, if not, then one can formulate the negation in a first-order way, expressing the fact that there is a set system of nonempty sets of dual VC-dimension strictly bounded by $q$ satisfying the $(p,q)$-property, without $n$ elements such that any set contains one of those elements, but then this must be true in a finite structure, contradicting \cref{fact:pq theorem}.
\end{remark}

\begin{proof}[Proof of \cref{thm:uniform definable pq}]
  First note that we may assume that all models of $T$ are infinite as otherwise we would be done by \cref{rem:pseudofinite case}. 
  
  By the usual coding tricks (see e.g., \cite[Theorem II.2.12(1)]{Sh-CT} and \cite[Lemma 2.5]{MR2963018}), it is enough to find formulas $\psi_i(y,w)$ for $i<n$ as in the statement such that for every $M,c$ some subset of these formulas satisfies the statement of the theorem. Indeed, if we find such formulas, then we can replace each $\psi_i$ by $\psi_i'(y,w,s,t_0,\dots,t_{n-1}):= ((s=t_i) \land \psi_i)$. Since all our models have size at least $2$, it is easy to see that the new formulas $\sequence{\psi_i'}{i<n}$ satisfy the requirements. 
  
  Suppose this is not true and fix $\phi(x,y,z),\psi(y,z),q$ and $p$ as above. For every choice of formulas $\Psi := \{\psi_0(y,w),\dots,\psi_{n-1}(y,w)\}$ (with $w$ being any tuple of variables) there is a model $M_{\Psi}$ and $c_{\Psi} \in M_{\Psi}^z$ witnessing that the conclusion does not hold: the pair $(\phi(x,y,c_\Psi),\psi(y,c_\Psi))$ has the $(p,q)$-property and $\vc^*(\phi(x,y,c_\Psi))<q$, but for all $s \subseteq n$ and $\sequence{d_i}{i\in s}$ from $M_\Psi^w$, either $\psi(y,c)$ is not equivalent to the disjunction $\bigvee_{i\in s}\psi_i(y,d_i)$ or for some $i\in s$, the set $\set{\phi(x,b)}{b\in \psi_i(M,d_i)}$ is inconsistent.

  Let $\mathcal{U}$ be an ultrafilter on the family $X$ of finite sets of formulas of the form $\xi(y,w)$ (for varying $w$) extending the filter generated by $\set{\Psi' \in X}{\Psi \subseteq \Psi'}$ for any $\Psi \in X$. Let $M^*$ be the ultraproduct $\prod_{\Psi \in X} M_{\Psi} / \mathcal{U}$. Finally, let $c^* \in M^*$ be the class of $\sequence{c_\Psi}{\Psi \in X}$. 

  Note that $(\phi(x,y,c^*), \psi(y,c^*))$ has the $(p,q)$-property and $\vc^*(\phi(x,y,c^*))<q$ since this is expressible in first-order. By \cref{lem:equivalence of non dividing formulas and pq}, it follows that for any $b \in \psi(\U,c^*)$, $\phi(x,b,c^*)$ does not fork over $M^*$. Now we want to apply \cref{cor:definable p q explanation}, but being consistent is not first-order expressible, so we have to be a bit more careful.
  
  By compactness and the proof of \cref{thm:the main theorem}, there are finitely many formulas $\psi_i(y,w),\theta_i(x,u,v),\zeta_i(u,v)$ without parameters, and $d_i \in (M^*)^w,e_i \in (M^*)^v$ for $i<n$ such that:
  \begin{enumerate}
    \item $\psi(y,c^*)$ is equivalent to the disjunction $\bigvee \psi_i(y,d_i)$.
    \item For each $i<n$, $(\theta_i(x,u,e_i),\zeta_i(u,e_i))$ has the $(p_i,q_i)$-property for some $\vc^*(\theta_i(x,u,e_i)) < q_i \leq p_i$ and let $n_i = \Npq(p_i,q_i)$
    \item For every set $B \subseteq \psi_i(M^*,d_i)$ of size $\leq n_i$ there is some $c \in \zeta_i(M^*,e_i)$ such that $M^* \vDash \forall x (\theta_i(x,c,e_i) \to \phi(x,b,c^*))$ for all $b\in B$.
  \end{enumerate}
  
  Let $\Psi = \{\psi_0,\dots,\psi_{n-1}\}$. By  \L{}o\'s' theorem, there is some $\Psi' \supseteq \Psi$ such that in $M_\Psi'$, there are $d_i',e_i'$ for which the above is true replacing $c^*$ by $c_{\Psi'}$. But then by applying \cref{prop:pq + implication => consistent} in $M_{\Psi'}$ we get that $\set{\phi(x,b,c_{\Psi'})}{b\in \psi_i(M_{\Psi'},d_i')}$ is consistent for all $i<n$, contradicting the choice of $M_{\Psi'}$.
\end{proof}

\subsection{A uniform version related to VC-density} \label{subsec:VC-codensity uniformity}
Matou\v sek's original formulation of \cref{fact:pq theorem} was in terms of VC-density and not in terms of (the less refined notion of) VC-dimension. We will explain this notion and prove a variant of \cref{thm:uniform definable pq} related to this notion.

Let $(X,\calF)$ be a set system. For any $n<\omega$, let $\pi_{\calF}(n) = \max\set{|\calF \cap A|}{A\subseteq X, |A| \leq n}$ where $\calF \cap A = \set{F\cap A}{F\in \calF}$. Similarly, define $\pi_\calF^*$ as $\pi_{\calF^*}$ (see \cref{fact:dual}).

\begin{definition}[VC-density]
Suppose that $(X,\calF)$ is a set system. The \defn{VC-density} of $\calF$, denoted by $\vcd(\calF)$, is $\limsup_{n\to \infty} \frac{\log \pi_\calF(n)}{\log n} \in \Rr \cup \{\infty\}$. 
The \defn{VC-codnsity} of $\calF$, denoted by $\vcd^*(\calF)$, is defined as $\vcd(\calF^*)$ (see \cref{fact:dual}).
\end{definition}

In other words, $\vcd(\calF)$ is the infimum of all non-negative real number $r$ such that $\pi_\calF(n)=O(n^r)$.

\begin{fact} 
  By the Sauer-Shelah lemma (see e.g., \cite[Lemma 6.4]{simon-NIP}), either $\pi_\calF(n) = 2^n$ for all $n<\omega$, which happens exactly when $\calF$ is not a VC-class, or $\pi_\calF(n) = O(n^{\vc(\calF)})$. Thus, $\vcd(\calF)\leq \vc(\calF)$.  
\end{fact}

Here is Matou\v sek's formulation of \cref{fact:pq theorem}:
\begin{fact} \cite[Theorem 4]{Matousek} \label{fact:Matousek original}
  Let $(X,\calF)$ be a set system and suppose that $\calF$ is a VC-class such that $\vcd^*(\calF)<q\leq p$ for some $q,p<\omega$, or even just $\pi^*_\calF(n) = o(n^q)$. Then there is some number $N$ such that for any finite $\mathcal{G} \subseteq \calF$, if $\mathcal{G}$ has the $(p,q)$-property, then for some $X_0 \subseteq X$ of size $N$, $X_0 \cap F \neq \emptyset$ for any nonempty $F \in \mathcal{G}$. 
\end{fact}

\begin{remark}
  To see why this formulation implies the one found in \cref{fact:pq theorem}, see \cite[Remark 7]{CS-extDef2}. 
\end{remark}

In the context of a complete $\mathcal{L}$-theory $T$, 
for a formula $\phi(x,y)$ define $\vcd(\phi)$ as $\vcd(\calF)$, 
where $\calF = \set{\phi(M,a)}{a\in M^y}$ for some/any model $M \vDash T$ 
(even though the VC-density of $\phi$ is not determined by a sentence, it is determined by a partial type describing $\pi_{\calF}(n)$ for each $n<\omega$). 
Similarly, define $\vcd^*(\phi)$ as the VC-codensity of $\calF$, i.e., $\vcd^*(\phi)=\vcd(\phi^{\opp})$.
In addition, let $\pi_\phi = \pi_\calF$ and $\pi_\phi^* = \pi^*_{\calF}$ ($=\pi_{\phiopp}$).

There are many theories where the VC-density of formulas have been computed, see \cite{VC-density1,VC-density2}. For example, \cite[Theorem 1.1]{VC-density1} states that in weakly-o-minimal theories, the VC-density of a formula $\phi(x,y)$ is bounded by $|y|$.

\begin{remark}
  A naive attempt at generalizing \cref{thm:uniform definable pq}, is the following statement:

  Suppose that $T$ is NIP, and $\phi'(x,y,z)$, $\psi'(y,z)$ are two formulas without parameters.  Then for any $q \leq p <\omega$ there is $n<\omega$ and formulas $\psi_0(y,w),\dots,\psi_{n-1}(y,w)$ such that the following hold.
  
  Suppose that $M\vDash T$ and $c \in M^z$. Let $\phi(x,y) = \phi'(x,y,c)$ and $\psi(y) = \psi'(y,c)$. If $(\phi,\psi)$ has the $(p,q)$-property and $\vcd^*(\phi(x,y)) < q$ then for some $d_0,\dots,d_{n-1} \in M^w$, $\psi(y)$ is equivalent to the disjunction $\bigvee_{i<n}\psi_i(y,d_i)$ and for each $i<n$, the set $\set{\phi(x,b)}{b\in \psi_i(M,d_i)}$ is consistent. 

  However, this is false. Indeed, let $T$ be as in \cref{rem:example where naive nf uniform version doesnt work}. Let $\psi'(y,z) = (y \mathrel{E} z)$ and $\phi'(x,y,z) = (x=y) \land \psi'(y,z)$. Let $q=p=1$. Suppose that $\psi_i(y,w)$ for $i<n$ are as in the statement. Let $M \vDash T$, let $c$ be such that $[c]_E$ is finite and of size $>n$ and let $\phi,\psi$ be as above. Then $\vcd^*(\phi(x,y))=0$ since $\set{\phi(M,a)}{a\in M}$ is finite. Clearly, the $(1,1)$-property holds for $(\phi,\psi)$. Thus, there are $d_i$ for $i<n$ as above, and we get the same contradiction as in \cref{rem:example where naive nf uniform version doesnt work}.
\end{remark}

A more sensible version is the following, which we will deduce from \cref{thm:uniform definable pq}.
\begin{corollary}\label{cor:Pablo's uniformity}
  Suppose that $T$ is NIP, and that $\phi(x,y)$, $\psi'(y,z)$ are two formulas without parameters. 
  Additionally, assume that $\vcd^*(\phi)< q \leq p < \omega$, or even just that $\pi^*_\phi(n)=o(n^q)$.
  Then there is $n<\omega$ and formulas $\psi_0(y,w),\dots,\psi_{n-1}(y,w)$ such that the following hold.
  
  Suppose that $M\vDash T$ and $c \in M^z$. Let $\psi(y) = \psi'(y,c)$. If $(\phi,\psi)$ has the $(p,q)$-property then for some $d_0,\dots,d_{n-1} \in M^w$, $\psi(y)$ is equivalent to the disjunction $\bigvee_{i<n}\psi_i(y,d_i)$ and for each $i<n$, the set $\set{\phi(x,b)}{b\in \psi_i(M,d_i)}$ is consistent. 
\end{corollary}

\begin{remark}
  \cref{cor:Pablo's uniformity} answers positively a question of And\'ujar-Guerrero \cite[Questions 3.7(2)]{Guerrero} where $T$ is assumed to be NIP, without restrictions on the VC-codensity.
\end{remark}

To prove this we will need the following observation, which was pointed out to us by Pablo And\'ujar-Guerrero. It is an improvement of \cref{lem:equivalence of non dividing formulas and pq}. 

\begin{lemma} \label{lem:pq for some q>vc*}
  Suppose that $(X,\calF)$ is a set system. Suppose that  $\pi^*_\calF(n) = o(n^q)$ and $q\leq p <\omega$. Then for any $0< q'<\omega$ there is $p'<\omega$ such that for any finite $\mathcal{G} \subseteq \mathcal{F}$, if $\mathcal{G}$ has the $(p,q)$-property then $\mathcal{G}$ has the $(p',q')$-property.
\end{lemma}

\begin{proof}
  Let $N$ be as in \cref{fact:Matousek original}. Given $q'<\omega$, let $p'=N\cdot (q'-1)+2$. The lemma follows by the choice of $N$ and the pigeonhole principle (it is $+2$ and not $+1$ to allow for the empty set to be in $\mathcal{G}$). 
\end{proof}

\begin{remark}
  The proof of \cref{lem:pq for some q>vc*} allows one to improve \cref{lem:equivalence of non dividing formulas and pq}: one can improve (2) by expressing $p$ as in the proof, and (3) can be improved by replacing $\vc^*(\phi)$ by $\vcd^*(\phi)$. 
\end{remark}

\begin{proof}[Proof of \cref{cor:Pablo's uniformity}]
   Let $q' = \vc^*(\phi)+1$ and let $p'$ be given by \cref{lem:pq for some q>vc*} for the family $\calF:=\set{\phi(M,a)}{a \in M^y}$. Now apply \cref{thm:uniform definable pq} with $\phi':= \phi$, $\psi'$, $q'$ and $p'$ to get $\psi_0(y,w),\dots,\psi_{n-1}(y,w)$ as in there. Then these formulas satisfy the desired conclusion: indeed, if $c \in M^z$ and $\psi(y):=\psi'(y,c)$ are such that $(\phi,\psi)$ have the $(p,q)$-property, then by the choice of $p'$, it has the $(p',q')$-property, and hence there are $d_0,\dots,d_{n-1}$ as required. 
\end{proof}

\section{Final thoughts and questions} \label{sec:final thoughts}

The following question remains open. 
\begin{question}
  In \cref{thm:the main theorem}, is it enough to assume that $\phi(x,y)$ is NIP instead of the whole theory?
\end{question}

A positive answer is conjectured in \cite[Conjecture 2.15]{simon-invariant}.

\begin{remark}
  Essentially, there are two places where the proof of \cref{thm:the main theorem} uses the fact that $T$ is NIP. The first one is the existence of strict non-forking global types which are coheirs (\cref{fact:non-coforking coheir}), but this holds assuming that $T$ is NTP$_2$ instead of NIP. The other place is in the very last part of the proof of \cref{prop:lemma as in Gareth Lotte} where we used that non-forking is the same as invariance over models (\cref{fac:forking in NIP}). However, this would still work if the implication formula $\forall x(\theta_m(x,z) \to \phi(x,y))$ was NIP. Thus, the theorem is still true provided that $T$ is NTP$_2$ and all formulas of the form  $\xi(z,y)=\forall x (\bigwedge_i \phi(x,z_i)^{\epsilon_i} \to \phi(x,y))$ are NIP. In fact, a close inspection of the proof yields a bound on the length of the tuple $z$ in terms of $\vc(\phi)$.  
\end{remark}

The following is natural in light of \cref{thm:uniform definable pq}:
\begin{problem}
  Suppose that $T$ is NIP, and fix two formulas $(\phi'(x,y,z),\psi'(y,z))$. Find an effective bound on the function mapping $(p,q)$ for $p \geq q$ to $n$ as in \cref{thm:uniform definable pq}, and explore how this function depends on $(\phi',\psi')$.
\end{problem}

Now we want to discuss when general theories satisfy the conclusion of \cref{thm:the main theorem}.

In the next definition, due to Adler\cite{Adler}, $\card$ denotes the class of all cardinals.
\begin{definition}
  We say that non-forking is  \defn{bounded} if there is a function $f : \card \to \card$ such that for any set $C$ and every $p\in S(C)$, the number of global non-forking extensions of $p$ is bounded by $f(|C|+|T|)$. 
\end{definition}

\begin{fact}\cite[Corollary 38]{Adler}
  Non-forking is bounded iff non-forking is equivalent to Lascar-non-splitting, that is: a global type $p$ does not fork over a model $M$ iff for any formula $\phi(x,y)$ and any $c,d$ such that $c\equiv_M d$, $\phi(x,c) \in p$ iff $\phi(x,d) \in p$. 
\end{fact}

Thus, from \cref{fac:forking in NIP} we get that:
\begin{fact}
  If $T$ is NIP then non-forking is bounded. 
\end{fact}

Say that non-forking is \defn{strongly bounded} if whenever $M \vDash T$, if $\phi(x,b)$ does not fork over $M$, then there is some $\psi(y) \in \tp(b/M)$ such that $\set{\phi(x,b')}{b' \in \psi(M)}$ is consistent. 

Justifying the name, we have: 
\begin{proposition}
  If non-forking is strongly bounded then non-forking is bounded.
\end{proposition}

\begin{proof}
  Suppose that non-forking is not bounded. Then there is a model $M \vDash T$, a formula $\phi(x,y)$ over $M$ and $c,d \in M^y$ such that $\phi(x,c) \land \neg \phi(x,d) \in p$ while $c \equiv_M d$. Recall that $c,d$ have Lascar distance at most 2: for some $e$, both $(c,e)$ and $(d,e)$ start an indiscernible sequence (by e.g., letting $e$ realize a global coheir extension of $\tp(c/M)$ over $Mcd$). So there is an $M$-indiscernible sequence $\sequence{c_i}{i<\omega}$ such that $\phi(x,c_0) \land \neg \phi(x,c_1)$ does not fork over $M$. Let $q(y_0,y_1)= \tp(c_0,c_1/M)$ and let $\psi(x,y_0,y_1)=\phi(x,y_0)\land \neg \phi(x,y_1)$. Then clearly $(c_0 c_1), (c_1 c_2) \vDash q$ but $\psi(x,c_0,c_1), \psi(x,c_1,c_2)$ is inconsistent, and in particular there is no formula $\psi(y) \in q$ witnessing strong boundedness.
\end{proof}

\begin{corollary}
  If non-forking is strongly bounded and is NTP$_2$, then $T$ is NIP. 
\end{corollary}

\begin{proof}
  This follows directly from \cite[Theorem 4.3]{forkingNTP2}.
\end{proof}

\begin{example}
  There is an IP (and thus TP$_2$) theory $T$ with strongly bounded non-forking. Indeed, by \cite[Corollary 5.24]{Sh1007}, there is such a theory where a global type $p$ does not fork over a model $M$ iff $p$ is finitely satisfiable in $M$. 

  Now, suppose that $\phi(x,b)$ does not fork over a model $M$. Then $\phi(M,b) \neq \emptyset$, so let $m\in M^x$ be such that $\phi(m,b)$ holds. Let $\psi(y) = \phi(m,y)$. Then clearly the set $\set{\phi(x,b')}{M\vDash \psi(b')}$ is consistent (realized by $m$) as required. 
\end{example}

\begin{question}
  Suppose that non-forking is bounded. Is it also strongly bounded?
\end{question}

\bibliographystyle{alpha}
\bibliography{definablepq}

\end{document}